\documentclass[11pt,openbib]{amsart}

\usepackage{amsmath}
\usepackage{amssymb}
\usepackage{amsthm}
\usepackage{amscd}
\usepackage{mathrsfs}

\newcommand{\ri}{\mathfrak{o}}
\newcommand{\mi}{\mathfrak{p}}

\newcommand{\Sch}[1]{\mathcal{C}_c^\infty({#1})}

\newcommand{\Z}{\mathbf{Z}}
\newcommand{\C}{\mathbf{C}}

\newcommand{\e}{\sqrt{\epsilon}}
\newcommand{\p}{\varpi}

\newcommand{\U}{\mathrm{U}}

\newcommand{\supp}{\operatorname{supp}}

\def\Section#1{\section{#1}\setcounter{equation}{0}}

\theoremstyle{plain}
\newtheorem{thm}[equation]{Theorem}
\newtheorem{lem}[equation]{Lemma}
\newtheorem{prop}[equation]{Proposition}
\newtheorem{cor}[equation]{Corollary}
\newtheorem{conj}[equation]{Conjecture}
\theoremstyle{definition}
\newtheorem{defn}[equation]{Definition}

\title{On epsilon factors attached to supercuspidal representations of unramified $\mathrm{U}(2,1)$}
\author{Michitaka Miyauchi}
\address{
Department of Mathematics, Faculty of Science\\
Kyoto University\\
Oiwake Kita-Shirakawa Sakyo Kyoto 606-8502 JAPAN
}
\email{miyauchi@math.kyoto-u.ac.jp}
\begin{document}
\maketitle
\pagestyle{myheadings}
\markboth{}{}

\begin{abstract}
Let $G$ be the unramified unitary group in three variables
defined over a $p$-adic field $F$ with $p \neq 2$.
Gelbart, Piatetski-Shapiro and Baruch attached
zeta integrals of Rankin-Selberg type
to irreducible generic representations of $G$.
In this paper,
we formulate a conjecture on
$L$- and $\varepsilon$-factors defined through zeta integrals
in terms of local newforms for $G$,
which is an analogue of the result by Casselman and Deligne for 
$\mathrm{GL}(2)$.
We prove our conjecture for the generic supercuspidal
representations of $G$.
\end{abstract}

\section*{Introduction}
Local newforms play an important role
in the theory of automorphic forms.
Casselman and Deligne established the local newform theory for $\mathrm{GL}(2)$,
which can be stated as follows.
Let $F$ be a non-archimedean local field of characteristic zero
with ring of integers $\ri_F$
and its maximal ideal $\mi_F$.
Let $\psi_F$ be a non-trivial additive character of $F$
with conductor $\ri_F$.
The local counterpart of a level subgroup of $\mathrm{GL}_2(F)$
is defined by
\[
\Gamma_0(\mi_F^n) 
= 
\left(
\begin{array}{cc}
\ri_F & \ri_F \\
\mi_F^n & 1+\mi_F^n
\end{array}
\right)^\times,
\]
for $n \geq 0$.
For each  irreducible admissible
representation $(\pi, V)$ of $\mathrm{GL}_2(F)$,
we
define the subspace
\[
V(n) = \{v \in V\, |\, \pi(k)v = v,\ k \in \Gamma_0(\mi_F^n)\}
\]
of $V$.
Then 
the following 
theorem holds:
\begin{thm}[\cite{Casselman}]\label{thm:gl2}
Let $(\pi, V)$ be an irreducible generic 
representation of $\mathrm{GL}_2(F)$.

(i) There exists a non-negative integer $n$ such that
$V(n) \neq \{0\}$.

(ii) Put $c(\pi) = \min \{n\, |\, V(n) \neq \{0\}\}$.
Then the space $V(c(\pi))$ is one-dimensional.

(iii)
The $\varepsilon$-factor $\varepsilon(s, \pi, \psi_F)$ of $\pi$
 is a constant multiple of $q_F^{-c(\pi)s}$,
 where $q_F$ is the cardinality of the residue field of $F$.
\end{thm}
We call
the integer $c(\pi)$ {\it the conductor of} $\pi$
and $V(c(\pi))$ {\it the space of newforms for} $\pi$.
Another important property of newforms
is that the
zeta integral of a newform expresses the $L$-factor of a
representation.

\begin{thm}[\cite{Deligne}]\label{thm:D}
Let $\pi$ be 
an irreducible generic 
representation of $\mathrm{GL}_2(F)$
and let
$W$ be the newform in the  Whittaker model of $\pi$.
Then the corresponding Jacquet-Langlands's
zeta integral 
$Z(s, W)$
attains the $L$-factor of $\pi$.
\end{thm}

Similar results were obtained by Jacquet, Piatetski-Shapiro and Shalika \cite{JPSS}
for $\mathrm{GL}_n(F)$.
Recently, Roberts and Schmidt \cite{RS}
established a theory of local newforms 
for irreducible representations of $\mathrm{GSp}_4(F)$
whose central characters are trivial.
Our main concern is to construct a local newform theory for 
unramified unitary group $\mathrm{U}(2,1)$.

The aim of this paper is twofold.
The first one is 
to formulate an analogue of Theorem~\ref{thm:D} for
Gelbart and Piatetski-Shapiro's
zeta integrals for unramified $\mathrm{U}(2,1)$.
Gelbart and Piatetski-Shapiro \cite{GPS}
initiated 
a theory of Rankin-Selberg integrals for $\mathrm{U}(2,1)$,
and
Baruch \cite{Baruch} developed it.
For a given  irreducible generic representation
of $\mathrm{U}(2,1)$,
they attached a zeta integral involving 
a Whittaker function and  a Schwartz function on $F^2$.
The $L$-factor attached to an irreducible generic representation
of $\mathrm{U}(2,1)$ is defined 
as the greatest common divisor of the zeta integrals.
For unramified $\mathrm{U}(2,1)$,
the author \cite{M} introduced a family of open compact subgroups
and
defined newforms for generic representations $\pi$.
We conjecture that the $L$-factor of $\pi$
is represented by
the zeta integral 
when the Whittaker function  is associated to the newform
and the Schwartz function is the characteristic function of a certain lattice in $F^2$
(Conjecture~\ref{conj:main}).
In addition,
our conjecture includes
the relation between $\varepsilon$-factors
and conductors.
(Theorem~\ref{thm:main}).
We note that there is a related work by
Koseki and Oda \cite{OdaKoseki}
for archimedean situation.

The second aim of this paper is
to show that our conjecture holds 
for the generic supercuspidal representations.
The zeta integral involving the newform
and the characteristic function of a certain lattice
in $F^2$
is decomposed into a product of 
$L_E(s, 1)$
and a more simple zeta integral of the newform,
which is close to that for $\mathrm{GL}(2)$
(Proposition~\ref{prop:zeta1}).
Here $E$ is the unramified quadratic extension over $F$
and
$L_E(s, 1)$ stands for the $L$-factor of the trivial 
representation of $\mathrm{GL}_1(E)$.
To 
compute the latter zeta integral,
we follow the method by Roberts and Schmidt
for $\mathrm{GSp}(4)$.
In \cite{RS},
they utilized
Hecke operators
acting on the space of newforms
to obtain
a key formula of the values of 
the Whittaker function associated to 
the newform at diagonal matrices
in terms of Hecke eigenvalues.
The zeta integral of the newform is  determined by this formula.
In their theory,
the assumption on the central character
is crucial.
We apply their method to 
representations of unramified $\mathrm{U}(2,1)$
whose conductors differ from those of their central characters
to obtain an explicit formula of zeta integrals of newforms 
in terms of Hecke eigenvalues (Proposition~\ref{prop:zeta}).
All the supercuspidal representations 
satisfy this assumption on the central characters,
so our conjecture holds for them.

We summarize the contents of this paper.
In section~\ref{sec:pre}, 
we fix the basic notation for representations of the unramified unitary group
in three variables
and recall the notion of its newforms.
In section~\ref{sec:RS},
we recall from \cite{Baruch} the theory
of Rankin-Selberg convolution for unramified $\mathrm{U}(2,1)$
and do some computation 
relating to newforms.
In section~\ref{sec:conj},
we give Conjecture~\ref{conj:main},
which says that
zeta integrals of newforms compute the $L$-factors of representations.
In section~\ref{sec:hecke},
we introduce the Hecke operator
and the level lowering operator,
and give their explicit description.
Using these operators,
we get a formula of 
zeta integrals of newforms in terms of Hecke eigenvalues.
In section~\ref{pf},
we show that our conjecture is true for all generic supercuspidal representations
of $G$.

Unfortunately, our conjecture is still open 
for non-supercuspidal representations.
Even for such representations,
under the condition on the central characters,
we may show the validity of the conjecture,
which will be treated in our ongoing work.
It is also an important problem to 
compare $L$- and $\varepsilon$-factors defined through
zeta integrals 
with those of $L$-parameters.

\medskip
\noindent
{\bf Acknowledgements} \
The author would like to thank Yoshi-hiro Ishikawa for
his advice and comments
and
Takuya Yamauchi for helpful discussions.

\Section{Preliminaries}\label{sec:pre}
In subsection~\ref{subsec:notation},
we fix our notation for
unramified $\mathrm{U}(2,1)$,
which is used in this paper.
In subsection~\ref{subsec:newform},
we recall from \cite{M} the definition and
basic properties of newforms for
unramified $\mathrm{U}(2,1)$.
\subsection{Notations}\label{subsec:notation}
Let $F$ be a non-archimedean local field of characteristic zero,
$\ri_F$ its ring of integers,
$\mi_F = \p_F \ri_F$ the maximal ideal in $\ri_F$.
We write $q = q_F$ for the cardinality of 
$\ri_F/\mi_F$.
Let $|\cdot|_F$ denote the absolute value of $F$
normalized so that $|\p_F|_F = q_F^{-1}$.
We use the analogous notation 
for any non-archimedean local field.
Throughout this paper,
we  assume that
the residual characteristic of 
$F$ is odd.

Let $E = F[\e]$ be the unramified quadratic  extension over $F$,
where $\epsilon$ is a non-square unit in $\ri_F$.
We know that $q_E = q^2$ and
$\p_F$ is a uniformizer of $E$.
So we abbreviate $\p = \p_F$.
We set
$G  = 
\{ g \in \mathrm{GL}_3(E)\ |\ 
{}^t \overline{g} Jg = J \}$,
where ${}^-$ is the non-trivial element in $\mathrm{Gal}(E/F)$
and
\begin{eqnarray*}
J = 
\left(
\begin{array}{ccc}
0 & 0 &1\\
0 & 1 & 0\\
1 & 0 & 0
\end{array}
\right).
\end{eqnarray*}
Then $G$ is a realization of the $F$-points of unramified $\U(2,1)$ defined 
over $F$.

Let $B$ be the Borel subgroup of $G$ consisting of the upper triangular elements
in $G$ with
Levi subgroup $T$ of diagonal matrices in $G$
and unipotent radical $U$.
We write $\hat{U}$ for the  opposite of $U$:
\begin{eqnarray*}
& U = \left\{
u(x, y) =
\left(
\begin{array}{ccc}
1 & x & y\\
0 & 1 & -\overline{x}\\
0 & 0 & 1
\end{array}
\right)\, 
\Bigg|\,
x, y \in E,\
y+\overline{y} +x\overline{x} = 0
\right\},\\
& \hat{U} = \left\{ \hat{u}(x,y) =
\left(
\begin{array}{ccc}
1 & 0 & 0\\
x & 1 & 0\\
y & -\overline{x} & 1
\end{array}
\right)\, 
\Bigg|\,
x, y \in E,\
y+\overline{y} +x\overline{x} = 0
\right\}.
\end{eqnarray*}
We shall identify the subgroup
\begin{eqnarray*}
H = \left\{
\left(
\begin{array}{ccc}
a & 0 & b\\
0 & 1 & 0\\
c & 0 & d
\end{array}
\right) \in G
\right\}
\end{eqnarray*}
of $G$ with 
$\U(1,1)(E/F)$.
We  set $B_H = B\cap H$, $U_H = U\cap H$
and $T_H = T\cap H$:
\begin{eqnarray*}
& T_H = \left\{
t(a) = \left(
\begin{array}{ccc}
a & 0 & 0\\
0 & 1 & 0\\
0 & 0 & \overline{a}^{-1}
\end{array}
\right)\, \Bigg|\,
a \in E^\times
\right\}.
\end{eqnarray*}
For $a \in E^\times$,
we put
$t(a) = \left(
\begin{array}{cc}
a & 0\\
0 & \overline{a}^{-1}
\end{array}
\right)$ and 
$d(a) = \left(
\begin{array}{cc}
a & 0\\
0 & 1
\end{array}
\right)$.
Then 
every element $h$ in $H = \U(1,1)(E/F)$
can be decomposed into
\begin{eqnarray}\label{eq:h}
h = t(b) d(\e) h_1 d(\e^{-1}),\
\end{eqnarray}
where $b \in E^\times$ and $h_1 \in \mathrm{SL}_2(F)$.

For a non-trivial additive character $\psi_E$ of $E$,
we define a character $\psi_E$ of $U$ by
\[
\psi_E(u(x, y)) = \psi_E(x),\ \mathrm{for}\
u(x, y) \in U.
\]
We say that  a smooth representation $(\pi, V)$ 
of $G$ is {\it generic} if
$\mathrm{Hom}_U(\pi, \psi_E) \neq \{0\}$.
Suppose that $(\pi, V)$ is irreducible and generic.
Then 
it is well-known that 
\[
\dim \mathrm{Hom}_U(\pi, \psi_E) = 1.
\]
By Frobenius reciprocity,
there exists a unique embedding of
$\pi$ into $\mathrm{Ind}_U^G \psi_E$
up to scalar.
The image $\mathcal{W}(\pi, \psi_E)$ of $V$ is called {\it the Whittaker model of} $\pi$.
By a non-zero functional $l \in \mathrm{Hom}_U(\pi, \psi_E)$,
which 
is called {\it the Whittaker functional},
we define {\it the Whittaker function} $W_v \in \mathcal{W}(\pi, \psi_E)$
associated to $v \in V$ by
\[
W_v(g) = l(\pi(g)v),\ g \in G.
\]

We may identify
the center $Z$ of $G$
with the norm-one subgroup $E^{1}$ of $E^\times$.
Under this identification,
we set open compact subgroups
of $Z$ as
\[
Z_0 = Z,\ Z_n = Z\cap (1+\mi_E^n),\ \mathrm{for}\ n \geq 1.
\]
For an irreducible admissible representation $\pi$ of $G$,
we denote by $\omega_\pi$ the central character of $\pi$.
We define {\it the conductor of} $\omega_\pi$
by
\begin{eqnarray*}
n_\pi = \mathrm{min}\{n \geq 0\, |\, \omega_\pi|_{Z_n} = 1\}.
\end{eqnarray*}

\subsection{Newforms}\label{subsec:newform}
For a non-negative integer $n$,
we define an open compact subgroup $K_n$ of $G$
by
\begin{eqnarray*}
K_n
=
\left(
\begin{array}{ccc}
\ri_E & \ri_E & \mi_E^{-n}\\
\mi_E^n & 1+\mi_E^n & \ri_E\\
\mi_E^n & \mi_E^n & \ri_E
\end{array}
\right)
\cap G.
\end{eqnarray*}
For a smooth representation $(\pi, V)$
of $G$,
we denote by $V(n)$ the space of $K_n$-fixed vectors in $V$,
namely,
\[
V(n) =\{ v \in V\, |\, \pi(k)v = v,\ k \in K_n\},\
n \geq 0.
\]

\begin{thm}[\cite{M} Theorems~2.8, 5.6, Corollary 5.5 (i)]\label{thm:new}
Suppose that $(\pi, V)$ is an irreducible generic representation 
of $G$.

(i) There exists a non-negative integer $n$ such that
$V(n) \neq \{0\}$.

(ii) Put $N_\pi = \mathrm{min}\{n\geq 0\, |\, V(n) \neq \{0\}\}$.
Then $\dim V(N_\pi) = 1$.

(iii) If $\pi$ is supercuspidal,
then we have
$N_\pi \geq 2$ and $N_\pi > n_\pi$.
\end{thm}

\begin{defn}[\cite{M} Definition 2.6]
Let $(\pi, V)$ be an irreducible generic representation 
of $G$.
We call 
$N_\pi$ {\it the conductor of $\pi$}
and 
$V(N_\pi)$ {\it the space of newforms for $\pi$}.
\end{defn}

We recall some properties of Whittaker functions associated 
to newforms.
Let $(\pi, V)$ be an irreducible generic representation of $G$.
For each $v \in V$,
we can regard $W_v|_{T_H}$
as a locally constant function on $E^\times$.
Along the lines of the Kirillov theory for $\mathrm{GL}(2)$,
we see that there exists an integer $n$ such that
$\supp W_v|_{T_H}$ is contained in $\mi_E^{-n}$.
Moreover, if $v$ is an element in $\langle \pi(u)w-w\ |\ u \in U,\ w \in V \rangle$,
then $W_v|_{T_H}$ is a compactly supported function on $E^\times$.
\begin{prop}[\cite{M} Corollary 4.6, Theorem 4.12]
\label{prop:new}
Suppose that $\psi_E$ 
has conductor $\ri_E$.
Let $\pi$ be an irreducible generic representation
of $G$ and let $v$ be a newform for $\pi$.

(i) The  function $W_v|_{T_H}$ is $\ri_E^\times$-invariant
and
its support is contained in $\ri_E$.

(ii)
Suppose that $N_\pi \geq 2$ and $N_\pi > n_\pi$.
Then $W_v(1) = 0$ if and only if $v = 0$.
\end{prop}

\Section{Rankin-Selberg convolution}\label{sec:RS}
In this section,
we recall from \cite{Baruch} the theory
of Rankin-Selberg convolution for unramified $\mathrm{U}(2,1)$
by Gelbart, Piatetski-Shapiro and Baruch,
and give some computation 
relating to newforms.
In four subsections,
we treat their zeta integrals, $L$-factors,
the functional equation, and $\varepsilon$-factors
respectively.
For an irreducible generic representation $\pi$ of $G$,
their zeta integral has the form $Z(s, W, \Phi)$,
where $W$ is a Whittaker function for $\pi$ and $\Phi$
is a Schwartz function on $F^2$.
In subsection~\ref{subsec:zeta1},
we see that $Z(s, W, \Phi)$ can be decomposed into
a product of a more simple zeta integral and the $L$-factor of 
the trivial representation of $E^\times$
when $W$ is associated to a newform for $\pi$
and $\Phi$ is the characteristic function of a certain lattice in $F^2$.
The $L$-factor attached to $\pi$
is defined as the greatest common divisor of 
the zeta integrals.
In subsection~\ref{subsec:zeta2},
we quote the result by Ishikawa
on the shape of
$L$-factors of supercuspidal representations.
In subsection~\ref{subsec:zeta3},
we recall their functional equation,
and remark that it is the \lq\lq right" functional equation.
In subsection~\ref{subsec:zeta4},
we show that their $\varepsilon$-factors
are monomial.

\subsection{Zeta integrals}\label{subsec:zeta1}
\setcounter{equation}{0}
Let $\Sch{F^2}$ be the space of locally constant, compactly supported functions on $F^2$.
For $\Phi \in \Sch{F^2}$ and $g \in \mathrm{GL}_2(F)$,
we define a function $g \Phi$ in $\Sch{F^2}$ by
\begin{eqnarray*}
(g\Phi)(x, y) = \Phi((x, y)g),\ (x, y) \in F^2.
\end{eqnarray*}
We normalize the Haar measure on $F^\times$ so that
the volume of $\ri_F^\times$ is one.
For $\Phi \in \Sch{F^2}$ and $g \in \mathrm{GL}_2(F)$,
we define a function $z(s, g, \Phi)$ on $\C$ by
\begin{eqnarray*}
z(s, g, \Phi)
=
\int_{F^\times}(g\Phi)(0, r) |r|_E^s d^\times r,\ s \in \C.
\end{eqnarray*}

For any subset $S$ of $F^2$,
we denote by $\mathrm{ch}_{S}$
the characteristic function of $S$.
We set $\Phi_n = \mathrm{ch}_{\mi_F^n \oplus \ri_F}$,
for each integer $n$.
We define
the $L$-factor 
$L_E(s, \chi)$ of
a quasi-character $\chi$ of $E^\times$
as usual:
\[
L_E(s, \chi)
= 
\left\{
\begin{array}{cl}
\displaystyle \frac{1}{1-\chi(\p)q^{-2s}}, & \mathrm{if}\ \chi\ \mathrm{is\ unramified};\\
1, & \mathrm{if}\ \chi\ \mathrm{is\ ramified}.\\
\end{array}
\right.
\]
\begin{lem}\label{lem:z}
For any $k \in 
\left(
\begin{array}{cc}
\ri_F & \mi_F^{-n}\\
\mi_F^n & \ri_F
\end{array}
\right)^\times$,
we have
$z(s, k, \Phi_n) = L_E(s, 1)$.
\end{lem}
\begin{proof}
Since the group
$
\left(
\begin{array}{cc}
\ri_F & \mi_F^{-n}\\
\mi_F^n & \ri_F
\end{array}
\right)^\times$ is the stabilizer of the lattice
$\mi_F^n \oplus \ri_F$,
we obtain
\begin{eqnarray*}
z(s, k, \Phi_n) & = & 
\int_{F^\times}(k\Phi_n)(0, r) |r|_E^s d^\times r
= \int_{F^\times}\Phi_n(0, r) |r|_E^s d^\times r\\
 & = & 
 \int_{\ri_F\cap F^\times} |r|_F^{2s} d^\times r
 = 
 \frac{1}{1-q^{-2s}}\\
 & = & 
 L_E(s, 1),
\end{eqnarray*}
as required.
\end{proof}

For $h \in H$ and $\Phi \in \Sch{F^2}$,
we set
\begin{eqnarray*}
f(s, h, \Phi)
=|b|_E^s z(s, h_1, \Phi),\ s \in \C,
\end{eqnarray*}
where $b \in E^\times$ and $h_1 \in \mathrm{SL}_2(F)$
are as in (\ref{eq:h}).
By \cite{Baruch} Lemma 2.5,
the definition of 
$f(s, h, \Phi)$ is independent of the choice of 
$b \in E^\times$ and $h_1 \in \mathrm{SL}_2(F)$.
Set $K_{n, H} = K_{n} \cap H$.
We may identify $K_{n, H}$ with
\[
\left(
\begin{array}{cc}
\ri_E & \mi_E^{-n}\\
\mi_E^n & \ri_E
\end{array}
\right) \cap \mathrm{U}(1,1)(E/F).
\]
Because $K_{n, H}$ is a good maximal compact subgroup of $H$,
we have an Iwasawa decomposition $H = U_H T_H K_{n, H}$.
We note that 
\[
f(s, t(a)k, \Phi)
= |a|_E^sf(s, k, \Phi),\ a \in E^\times, k \in K_{n, H}.
\]

\begin{lem}\label{lem:f}
For any $k \in K_{n, H}$,
we have
$f(s, k, \Phi_n) = L_E(s, 1)$.
\end{lem}
\begin{proof}
We can decompose
$k = t(b) d(\e) k_1 d(\e^{-1})$, where
$b \in \ri_E^\times$ and $k_1 \in \left(
\begin{array}{cc}
\ri_F & \mi_F^{-n}\\
\mi_F^n & \ri_F
\end{array}
\right)^\times \cap \mathrm{SL}_2(F)$.
So we obtain
$f(s, k, \Phi)
 =  |b|_E^s z(s, k_1, \Phi)
 =  L_E(s, 1)$
by Lemma~\ref{lem:z}.
\end{proof}

Let $\pi$ be an irreducible generic representation of $G$.
For $W \in \mathcal{W}(\pi, \psi_E)$
and $\Phi \in \Sch{F^2}$,
we define the zeta integral
\begin{eqnarray*}
Z(s, W, \Phi)
=
\int_{U_H\backslash H}W(h)f(s, h, \Phi) dh.
\end{eqnarray*}
By \cite{Baruch} Proposition 3.4,
the integral $Z(s, W, \Phi)$ absolutely
converges to a  function in $\C(q^{-2s})$ when  $\mathrm{Re}(s)$ is sufficiently large.
We normalize the Haar measures on $E^\times$
and $K_{n, H}$
so that 
the volumes of $\ri_E^\times$ and $K_{n, H}$ are one
respectively.
By the Iwasawa decomposition $H = U_H T_H K_{n, H}$
and the isomorphism $E^\times \simeq T_H ; a \mapsto t(a)$,
we obtain
\begin{eqnarray}\label{eq:zeta_decomp}
Z(s, W, \Phi)  = 
\int_{E^\times}\int_{K_{n, H}} W(t(a)k) f(s, t(a)k, \Phi)
|a|_E^{-1}
d^\times a dk.
\end{eqnarray}
We define another zeta integral of $W$ in $\mathcal{W}(\pi, \psi_E)$
by
\begin{eqnarray*}
Z(s, W)
= 
\int_{E^\times} W(t(a)) |a|_E^{s-1} d^\times a.
\end{eqnarray*}
By the proof of \cite{Baruch} Proposition 3.4,
$Z(s, W)$ also
converges absolutely to a  function in $\C(q^{-2s})$ if  $\mathrm{Re}(s)$ is enough large.
\begin{prop}\label{prop:zeta1}
Suppose that a Whittaker function $W$ 
for $\pi$
is fixed by $K_{n, H}$.
Then we have
\[
Z(s, W, \Phi_n)
= Z(s, W)L_E(s, 1).
\]
\end{prop}
\begin{proof}
We obtain
\begin{eqnarray*}
Z(s, W, \Phi_n) & = & 
\int_{E^\times}\int_{K_{n, H}} W(t(a)k) f(s, t(a)k, \Phi_n)
|a|_E^{-1}
d^\times a dk\\
& = & 
\int_{E^\times}\int_{K_{n, H}} W(t(a))|a|_E^{s-1} f(s, k, \Phi_n)
d^\times a dk\\
& = & 
L_E(s, 1)
\int_{E^\times} W(t(a))|a|_E^{s-1} d^\times a
\\
& = & 
Z(s, W) L_E(s, 1)
\end{eqnarray*}
by (\ref{eq:zeta_decomp})
and
Lemma~\ref{lem:f}.
\end{proof}
\subsection{$L$-factors}\label{subsec:zeta2}
Let $I_{\pi}$ be
the subspace of $\C(q^{-2s})$
spanned by $Z(s, W, \Phi)$ where $\Phi \in \Sch{F^2}$,
$W \in \mathcal{W}(\pi, \psi_E)$ and $\psi_E$
runs over the non-trivial additive characters of $E$.
As remarked in \cite{Baruch} p. 331, $I_{\pi}$ is a fractional ideal of $\C(q^{-2s})$.
So we can find a polynomial $P(X) \in \C[X]$ such that
$P(0) = 1$ and $1/P(q^{-2s})$ generates $I_{\pi}$.
We define the $L$-factor $L(s, \pi)$ of $\pi$ by
\begin{eqnarray*}
L(s, \pi) = \frac{1}{P(q^{-2s})}.
\end{eqnarray*}

The following proposition is due to Ishikawa.
But the statement is slightly modified.
\begin{prop}[\cite{Ishikawa} Theorem 4 (4)]\label{prop:L_sc}
Let $\pi$ be an irreducible generic supercuspidal representation
of $G$.
Then $L(s, \pi)$
equals to $1$ or $L_E(s, 1)$.
\end{prop}
\begin{proof}
It is enough to show that
the function $Z(s, W, \Phi)/L_E(s, 1)$ belongs to 
$\C[q^{-2s}, q^{2s}]$
for all $W \in \mathcal{W}(\pi, \psi_E)$, $\Phi \in \Sch{F^2}$
and non-trivial additive characters $\psi_E$ of $E$.
Since $W(h)$ and $f(s, h, \Phi)$ are right smooth with 
respect to $h \in H$,
the integral $Z(s, W, \Phi)$
is a linear combination of functions of the form
$Z(s, W')f(s, 1, \Phi')$, where 
$W' \in \mathcal{W}(\pi, \psi_E)$ and $\Phi' \in \Sch{F^2}$.
It follows from \cite{M} Propositions 4.1 (ii) and 4.7
that
$W'|_{T_H}$ is a compactly supported function on $T_H \simeq E^\times$.
Note that in \cite{M},
we assume that $\psi_E$ has conductor $\ri_E$,
but Proposition 4.7 holds for all $\psi_E$.
This implies that $Z(s, W')$ lies in $\C[q^{-2s}, q^{2s}]$.
Due to the theory of zeta integrals for $\mathrm{GL}(1)$,
we see that $f(s, 1, \Phi')/L_E(s, 1)$ belongs to $\C[q^{-2s}, q^{2s}]$.
This completes the proof.
\end{proof}

\subsection{The functional equation}\label{subsec:zeta3}
Let $\psi_F$ be a non-trivial additive character of $F$
with conductor
$\mi_F^{c(\psi_F)}$.
We choose the Haar measure on $F^2$ normalized so that
the volume of $\ri_F \oplus \ri_F$ is $q^{c(\psi_F)}$.
For each $\Phi \in \Sch{F^2}$,
we define the Fourier transform $\hat{\Phi}$
by
\begin{eqnarray*}
\hat{\Phi}(x, y)
= \int_{F^2} \Phi(u, v) \psi_F(yu-xv) dudv.
\end{eqnarray*}
One can check that $\hat{\hat{\Phi}} = \Phi$
for all $\Phi \in \Sch{F^2}$.

\begin{lem}\label{lem:z2}
Suppose that the conductor of $\psi_F$  
is $\ri_F$. Then
we have
\[
z(1-s, k, \hat{\Phi}_n) =  q^{-2n(s-1/2)}L_E(1-s, 1),
\]
for any $k \in 
\left(
\begin{array}{cc}
\ri_F & \mi_F^{-n}\\
\mi_F^n & \ri_F
\end{array}
\right)^\times$.
\end{lem}
\begin{proof}
It is easy to observe that $\hat{\Phi}_n = q^{-n}
\mathrm{ch}_{\ri_F \oplus \mi_F^{-n}}$.
Since $\hat{\Phi}_n$ is fixed by $k$,
we get
\begin{eqnarray*}
z(1-s, k, \hat{\Phi}_n) & = & q^{-n}
\int_{F^\times}(k\cdot \mathrm{ch}_{\ri_F \oplus \mi_F^{-n}})(0, r) |r|_E^{1-s} d^\times r\\
& = & q^{-n}
\int_{F^\times}\mathrm{ch}_{\ri_F \oplus \mi_F^{-n}}(0, r) |r|_E^{1-s} d^\times r\\
 & = &  q^{-n}
 \int_{\mi_F^{-n}\cap F^\times} |r|_E^{1-s} d^\times r\\
 & = & 
 q^{-n} |\p^{-n}|_E^{1-s}
 \int_{\ri_F\cap F^\times} |r|_E^{1-s} d^\times r\\
 & = & q^{-2n(s-1/2)}
 L_E(1-s, 1),
\end{eqnarray*}
as required.
\end{proof}

\begin{cor}\label{cor:f}
If the conductor of $\psi_F$  
is $\ri_F$, then
we have
$f(1-s, k, \hat{\Phi}_n) = q^{-2n(s-1/2)}L_E(1-s, 1)$,
for $k \in K_{n, H}$.
\end{cor}
\begin{proof}
Exactly same as the proof of Lemma~\ref{lem:f}.
\end{proof}

\begin{prop}\label{prop:zeta2}
Suppose that $\psi_F$ has conductor $\ri_F$.
If a Whittaker function 
$W$ in $\mathcal{W}(\pi, \psi_E)$
is fixed by $K_{n, H}$,
then we have
\[
Z(1-s, W, \hat{\Phi}_n)
= q^{-2n(s-1/2)} Z(1-s, W, \Phi_n).
\]
\end{prop}
\begin{proof}
By (\ref{eq:zeta_decomp})
and
Corollary~\ref{cor:f},
we obtain
\begin{eqnarray*}
Z(1-s, W, \hat{\Phi}_n) & = & 
\int_{E^\times}\int_{K_{n, H}} W(t(a)k) f(1-s, t(a)k, \hat{\Phi}_n)
|a|_E^{-1}
d^\times a dk\\
& = & 
\int_{E^\times}\int_{K_{n, H}} W(t(a))|a|_E^{-s} f(1-s, k, \hat{\Phi}_n)
d^\times a dk\\
& = & 
q^{-2n(s-1/2)} L_E(1-s, 1)
\int_{E^\times}W(t(a))|a|_E^{-s} 
d^\times a\\
& = & 
q^{-2n(s-1/2)}
Z(1-s, W) L_E(1-s, 1).
\end{eqnarray*}
So
we get
$Z(1-s, W, \hat{\Phi}_n)
= q^{-2n(s-1/2)} Z(1-s, W, \Phi_n)$
due to Proposition~\ref{prop:zeta1}.
\end{proof}

By \cite{Baruch} Corollary 4.8,
there exists a rational function $\gamma(s, \pi, \psi_F, \psi_E)$
in $q^{-2s}$ such that
\begin{eqnarray}\label{eq:fe}
\gamma(s, \pi, \psi_F, \psi_E)Z(s, W, \Phi)
=Z(1-s, W, \hat{\Phi}).
\end{eqnarray}
We note that this is the \lq\lq right" functional equation.
For an irreducible admissible representation $(\pi, V)$
of $G$,
we denote by $\widetilde{\pi}$ its contragradient representation,
and by $\overline{\pi}$ the representation 
of $G$ on $V$ defined by
\[
\overline{\pi}(g) = \pi(\overline{g}),\ g \in G.
\]
\begin{lem}\label{lem:bar}
Let $\pi$ be an irreducible admissible representation of $G$.
Then $\widetilde{\pi}$
is isomorphic to $\overline{\pi}$.
\end{lem}
\begin{proof}
We define a hermitian form $h$ on $E^3$
by
\[
h(v, w) = {}^t \overline{v}J w,\ v, w \in E^3.
\]
Then $G$ is just the group of isometries of $(E^3, h)$.
Let $\delta$ be the element in $\mathrm{Aut}_F E^3$
defined by
$\delta v = \overline{v}$ for $v \in E^3$.
Then it follows from \cite{MVW} p. 91 that
$\widetilde{\pi}$ is isomorphic to $\pi^\delta$
where $\pi^\delta (g) = \pi(\delta g \delta^{-1})$, for $g \in G$.
Since $\pi^\delta (g) = \overline{\pi}(g)$,
for $g \in G$,
the lemma follows.
\end{proof}

We further assume $\pi$ is generic.
For $W \in \mathcal{W}(\pi, \psi_E)$,
we set 
\[
\overline{W}(g) = W(\overline{g}),\ g \in G.
\]
By Lemma~\ref{lem:bar},
we see that $\overline{W}$ lies
in $\mathcal{W}(\overline{\pi}, \overline{\psi}_E) = 
\mathcal{W}(\widetilde{\pi}, \overline{\psi}_E)$,
where $\overline{\psi}_E$ is the character of $U$
given by $\overline{\psi}_E(u) = \psi_E(\overline{u})$, $u \in U$.
We define another Fourier transformation on $\Sch{F^2}$
as follows: 
\begin{eqnarray*}
{\Phi}^*(x, y)
= \int_{F^2} \Phi(u, v) \psi_F(yu+xv) dudv,\ \Phi \in \Sch{F^2}.
\end{eqnarray*}
The following lemma implies that
(\ref{eq:fe})
is the \lq\lq right" functional equation.
\begin{lem}\label{lem:fe}
For $W \in \mathcal{W}(\pi, \psi_E)$ and $\Phi \in \Sch{F^2}$,
we have
\[
Z(s, W, \hat{\Phi}) = Z(s, \overline{W}, \Phi^*).
\]
\end{lem}
\begin{proof}
We claim that $d(-1) \Phi^* = \hat{\Phi}$.
In fact, we have
\begin{eqnarray*}
(d(-1){\Phi}^*)(x, y)
& = & {\Phi}^*((x, y)d(-1))
= {\Phi}^*(-x, y)\\
& = & 
\int_{F^2} \Phi(u, v) \psi_F(yu-xv) dudv\\
& = & \hat{\Phi}(x, y),
\end{eqnarray*}
for $(x, y) \in F^2$.
Thus, for $g \in \mathrm{GL}_2(F)$,
we obtain
\begin{eqnarray*}
z(s, d(-1){g} d(-1), \Phi^*)
& = & 
\int_{F^\times}(d(-1)gd(-1)\Phi^*)(0, r) |r|_E^s d^\times r\\
& = & 
\int_{F^\times}(g\hat{\Phi})((0, r)d(-1)) |r|_E^s d^\times r\\
& = & 
\int_{F^\times}(g\hat{\Phi})(0, r) |r|_E^s d^\times r\\
& = & 
z(s, {g}, \hat{\Phi}).
\end{eqnarray*}

Next, we shall show that $f(s, \overline{h}, \Phi^*)
= f(s, {h}, \hat{\Phi})$ for all $h \in H$.
Suppose that $h \in H$
is written as
$h = t(b) d(\e) h_1 d(\e^{-1})$,
where $b \in E^\times$ and $h_1 \in \mathrm{SL}_2(F)$.
Then we get
\begin{eqnarray*}
\overline{h} & = & t(\overline{b}) d(\overline{\e})\overline{h}_1d(\overline\e^{-1})
 =   t(\overline{b}) d(-{\e}){h_1} d(-\e^{-1})\\
& =&  t(\overline{b}) d({\e})d(-1){h_1} d(-1) d(\e^{-1}).
\end{eqnarray*}
Since $d(-1){h_1} d(-1) $ lies in $\mathrm{SL}_2(F)$,
we have
$f(s, \overline{h}, \Phi^*)
=|\overline{b}|_E^s z(s, d(-1){h_1} d(-1), \Phi^*)
= |{b}|_E^s z(s, h_1, \hat{\Phi})
= f(s, {h}, \hat{\Phi})$.

Haar measure $dh$ on $U_H\backslash H$ satisfies
$dh = d\overline{h}$
since $U_H$ and $K_{n, H}$ are stable
under the action of $\mathrm{Gal}(E/F)$.
Therefore we get
\begin{eqnarray*}
Z(s, \overline{W}, \Phi^*)
& = & 
\int_{U_H\backslash H}\overline{W}(h)f(s, h, \Phi^*) dh
 = 
\int_{U_H\backslash H}{W}(\overline{h})f(s, h, \Phi^*) dh\\
& = & 
\int_{U_H\backslash H}{W}({h})f(s, \overline{h}, \Phi^*) dh
=
\int_{U_H\backslash H}{W}({h})f(s, {h}, \hat{\Phi}) dh\\
& = & 
Z(s, W, \hat{\Phi}).
\end{eqnarray*}
This is the asserted equation.
\end{proof}
Lemma~\ref{lem:fe}
tells that
the $L$-factor of $\pi$ coincides with that of 
$\widetilde{\pi}$.
\begin{prop}\label{prop:L_cont}
For any irreducible generic representation $\pi$ of $G$,
we have
$L(s, \pi) = L(s, \widetilde{\pi})$.
\end{prop}
\begin{proof}
It follows from Lemma~\ref{lem:fe}
that the space $I_\pi$ coincides with 
$I_{\widetilde{\pi}}$.
Now the assertion is obvious.
\end{proof}
\subsection{$\varepsilon$-factors}\label{subsec:zeta4}
The $\varepsilon$-factor $\varepsilon(s, \pi, \psi_F, \psi_E)$ 
of  an irreducible generic representation $\pi$ of $G$
is defined by
\begin{eqnarray*}
\varepsilon(s, \pi, \psi_F,  \psi_E)  = 
\gamma(s, \pi, \psi_F, \psi_E) \frac{L(s, \pi)}{L(1-s, \widetilde{\pi})}.
\end{eqnarray*}
Due to Proposition~\ref{prop:L_cont},
we have
\begin{eqnarray}\label{eq:epsilon}
\varepsilon(s, \pi, \psi_F,  \psi_E)  = 
\gamma(s, \pi, \psi_F, \psi_E) \frac{L(s, \pi)}{L(1-s, {\pi})}.
\end{eqnarray}

\begin{prop}\label{prop:mono}
The $\varepsilon$-factor $\varepsilon(s, \pi, \psi_F, \psi_E)$ 
is a monomial in $\C[q^{-2s}, q^{2s}]$
of the form
\[
\varepsilon(s, \pi, \psi_F, \psi_E)
= \pm q^{-2n(s-1/2)},
\]
with some $n \in \Z$.
\end{prop}
\begin{proof}
By definition,
there exist
$\Phi_i \in \Sch{F^2}$,
additive characters $\psi_{E, i}$ of $E$
and
$W_i \in \mathcal{W}(\pi, \psi_{E, i})$ 
($1 \leq i\leq k$)
such that
\begin{eqnarray}\label{eq:rep1}
\sum_{i = 1}^k \frac{Z(s, W_i, \Phi_i)}{L(s, \pi)} = 1.
\end{eqnarray}
By the proof of \cite{Baruch} Lemma 4.9,
there exists
$W'_i \in \mathcal{W}(\pi, \psi_E)$
such that $Z(s, W_i, \Phi_i)
= q^{-2sm_i}Z(s, W'_i, \Phi_i)$,
for some $m_i \in \Z$.
From (\ref{eq:rep1}),
we have the following expression of the $\varepsilon$-factor:
\begin{eqnarray*}
\varepsilon(s, \pi, \psi_F, \psi_E)
& = & 
\varepsilon(s, \pi, \psi_F, \psi_E)
\sum_{i = 1}^k \frac{q^{-2sm_i} Z(s, W'_i, {\Phi}_i)}{L(s, \pi)}\\
& = & 
\sum_{i = 1}^k \frac{q^{-2sm_i} Z(1-s, W'_i, \hat{\Phi}_i)}{L(1-s, \pi)}.
\end{eqnarray*}
The second equality is a consequence of 
(\ref{eq:fe}) and (\ref{eq:epsilon}).
This implies that $\varepsilon(s, \pi, \psi_F, \psi_E)$
is a polynomial in $q^{-2s}$ and $q^{2s}$.

By the above expression of the $\varepsilon$-factor,
we get
\begin{eqnarray*}
\varepsilon(s, \pi, \psi_F, \psi_E)
\varepsilon(1-s, \pi, \psi_F, \psi_E)
& = & 
\varepsilon(s, \pi, \psi_F, \psi_E)
\sum_{i = 1}^k \frac{q^{-2(1-s)m_i} Z(s, W'_i, \hat{\Phi}_i)}{L(s, \pi)}\\
& =& 
\sum_{i = 1}^k \frac{q^{-2(1-s)m_i} Z(1-s, W'_i, {\Phi}_i)}{L(1-s, \pi)}\\
& = & 1.
\end{eqnarray*}
In the second equality,
we use 
the functional equation and
$\hat{\hat{\Phi}}_i = \Phi_i$.
The last equality is a consequence of (\ref{eq:rep1}).
Now the assertion follows by standard arguments.
\end{proof}

\Section{Conjecture on local newforms}\label{sec:conj}
We give the following conjecture on 
zeta integrals of newforms.
\begin{conj}\label{conj:main}
Let $E$ be the unramified quadratic extension over a non-archimedean 
local field $F$ of characteristic zero and of odd residual characteristic.
We fix an additive character $\psi_E$ of 
$E$ 
with conductor $\ri_E$.
Let $\pi$ be an irreducible generic representation of $\mathrm{U}(2,1)(E/F)$.
Then there exists a newform $v$ for $\pi$ 
which satisfies
\[
Z(s, W_v, \Phi_{N_\pi}) = L(s, \pi),
\]
where $N_\pi$ is the conductor of $\pi$ and
$\Phi_{N_\pi}$ is the characteristic function of 
$\mi_F^{N_\pi} \oplus \ri_F$.
\end{conj}

Here is our main theorem,
which will be proved in section~\ref{pf}.
\begin{thm}\label{thm:main}
Conjecture~\ref{conj:main} holds for
any irreducible generic supercuspidal representations of $\mathrm{U}(2,1)(E/F)$.
\end{thm}

If Conjecture~\ref{conj:main} is true, then
we  obtain a formula of the $\varepsilon$-factors,
which says that the exponents of $q^{-2s}$ of the 
$\varepsilon$-factors for generic representations agree with  their conductors.
\begin{thm}\label{thm:main2}
Let $E$ be the unramified quadratic extension over a non-archimedean 
local field $F$ of characteristic zero and of odd residual characteristic.
Let $\psi_E$ and $\psi_F$ be additive characters of 
$E$ and $F$ with conductors $\ri_E$ and $\ri_F$ respectively.
Assume that Conjecture~\ref{conj:main} is true
for
an irreducible generic representation $\pi$
of $\mathrm{U}(2,1)(E/F)$.
Then we have
\begin{eqnarray*}
\varepsilon(s, \pi, \psi_F, \psi_E)
= q_E^{-N_\pi(s-1/2)},
\end{eqnarray*}
where $N_\pi$ is the conductor of $\pi$
and $q_E$ is the cardinality of the residue field of $E$.
\end{thm}
\begin{proof}
By assumption,
there exists a newform $v$ for $\pi$ such that
$Z(s, W_v, \Phi_{N_\pi}) = L(s, \pi)$.
By Proposition~\ref{prop:zeta2} and the assumption,
we have 
\begin{eqnarray*}
Z(1-s, W_v, \hat{\Phi}_{N_\pi})
& = & q^{-2N_\pi(s-1/2)} Z(1-s, W_v, {\Phi}_{N_\pi})\\
& = &q^{-2N_\pi(s-1/2)} L(1-s, {\pi}).
\end{eqnarray*}
So we get
\begin{eqnarray*}
\varepsilon(s, \pi, \psi_F, \psi_E) & = &
\gamma(s, \pi, \psi_F, \psi_E) \frac{L(s, \pi)}{L(1-s,{\pi})}\\
& = & 
\frac{Z(1-s, W, \hat{\Phi}_{N_\pi})}{Z(s, W, \Phi_{N_\pi})}
\frac{L(s, \pi)}{L(1-s, {\pi})}\\
& = & 
q^{-2N_\pi(s-1/2)}.
\end{eqnarray*}
This proves the theorem.
\end{proof}

\Section{Zeta integrals of newforms}\label{sec:hecke}
In this section,
we give a formula of zeta integrals of newforms
for generic representations $(\pi, V)$ of $G$ which satisfy 
$N_\pi \geq 2$ and $N_\pi > n_\pi$.
In subsection~\ref{subsec:hecke1},
we introduce the Hecke operator $T$ on $V(n)$
and give its explicit description.
If $n = N_\pi$, then this operator is scalar
and has the Hecke eigenvalue.
In subsection~\ref{subsec:hecke2},
we consider the level lowering operator $\delta: V(n) \rightarrow V(n-1)$
and give its explicit formula
when $n \geq 2$ and $n > n_\pi$.
Combining these results,
in subsection~\ref{subsec:hecke3},
we get 
a recursion
formula of the values of the Whittaker functions associated to newforms
at diagonal matrices (Lemma~\ref{lem:rec}),
which gives an explicit formula of zeta integrals of newforms 
in terms of Hecke eigenvalues
(Proposition~\ref{prop:zeta}).
The results in this section are strongly inspired by those in \cite{RS} 
sections 6 and 7.
\subsection{Hecke operator}\label{subsec:hecke1}
From now on,
we fix a non-trivial additive character $\psi_E$ of 
$E$ 
with conductor $\ri_E$.
Put 
\[
\zeta
=
\left(
\begin{array}{ccc}
\p & & \\
& 1 & \\
& & \p^{-1}
\end{array}
\right).
\]
Let $(\pi, V)$ be an irreducible generic representation
of $G$
and $n$ a non-negative integer.
We define {\it the Hecke operator} $T$ on $V(n)$ by
\[
Tv = \frac{1}{\mathrm{vol}(K_n)}\int_{K_n \zeta K_n} \pi(k) v  dk,\ v \in V(n).
\]
Using the bijection
$K_n / K_n \cap \zeta K_n \zeta^{-1}
\simeq K_n \zeta K_n /K_n;
k (K_n \cap \zeta K_n \zeta^{-1})\mapsto k\zeta K_n$,
we can write $Tv$ as 
\begin{eqnarray}\label{eq:T}
Tv = \sum_{k \in K_n / K_n \cap \zeta K_n \zeta^{-1}}\pi(k\zeta)v.
\end{eqnarray}
We set
\[
t_n
= \left(
\begin{array}{ccc}
 & & \p^{-n}\\
 & 1 & \\
 \p^n & & 
\end{array}
\right) \in K_n.
\]
\begin{lem}\label{lem:hecke_coset}
Suppose that $n \geq 1$.
Then a complete set of representatives for 
$K_n / K_n \cap \zeta K_n \zeta^{-1}$
is given by
\[
t_n u(y, z\e-y\overline{y}/2)\ {and}\ u(a, b\e-a\overline{a}/2),
\]
where $y, a \in \ri_E/\mi_E$, $z \in \mi_F^{1-n}/\mi_F^{2-n}$
and $b \in \mi_F^{-n}/\mi_F^{2-n}$.
\end{lem}
\begin{proof}
Note that
\[
 K_n \cap \zeta K_n \zeta^{-1}
 =
 \left(
\begin{array}{ccc}
\ri_E & \mi_E &  \mi_E^{2-n}\\
\mi_E^n & 1+ \mi_E^n &  \mi_E\\
\mi_E^n & \mi_E^n &  \ri_E
\end{array}
\right)\cap G.
\]
Set 
\[
K'
 =
 \left(
\begin{array}{ccc}
\ri_E & \ri_E &  \mi_E^{1-n}\\
\mi_E^n & 1+ \mi_E^n &  \ri_E\\
\mi_E^n & \mi_E^n &  \ri_E
\end{array}
\right)\cap G.
\]
Clearly, the following two claims assert the lemma:
\begin{enumerate}
\item[(i)]
A complete set of representatives for 
$K_n / K'$
is given by the $q+1$ elements
$t_n$ {and} $u(0, x\e)$,
where 
$x \in \mi_F^{-n}/\mi_F^{1-n}$.

\item[(ii)]
We can take
a complete set of representatives for 
$K' / K_n \cap \zeta K_n \zeta^{-1}$
as the $q^3$ elements
$u(y, z\e-y\overline{y}/2)$,
where $y \in \ri_E/\mi_E$, $z \in \mi_F^{1-n}/\mi_F^{2-n}$.
\end{enumerate}
We shall prove these claims.
It is obvious that elements in (i) and (ii)
belong to pairwise distinct cosets in 
$K_n / K'$ and $K' / K_n \cap \zeta K_n \zeta^{-1}$
respectively.

(i)
We denote by $g_{ij}$ the $(i, j)$-entry of $g \in M_3(E)$.
Let $k \in K_n$.
If $k_{33} \in \mi_E$,
then we get $t_n k \in K'$, and hence $k \in t_n K'$.
Suppose that $k_{33} \in \ri_E^\times$.
Since $k$ lies in $G$,
we have $k_{13}\overline{k}_{33}+k_{23}\overline{k}_{23}
+k_{33}\overline{k}_{13} = 0$.
This implies 
$k_{13}\overline{k}_{33}
+k_{33}\overline{k}_{13} \in \ri_F$,
so that
$k_{13}\overline{k}_{33} \in \ri_F \oplus \mi_F^{-n}\e$.
By the assumption $k_{33} \in \ri_E^\times$,
we have
$k_{13}{k}_{33}^{-1} \in \ri_F \oplus \mi_F^{-n}\e$.
Thus there exists $x \in  \mi_F^{-n}$
such that 
$k_{13}{k}_{33}^{-1} -x\e \in \ri_F$.
Using the assumption again,
we get 
$k_{13} -x\e k_{33}\in \ri_E \subset \mi_E^{1-n}$.
This implies $k \in u(0, x\e)K'$.

(ii)
If $n = 1$,
then $K'$ lies in the standard Iwahori subgroup of $G$.
One can see that $K'$ has an Iwahori decomposition
$K' = (K'\cap \hat{U})(K'\cap T)(K'\cap U)$.
The assertion follows
because $(K'\cap \hat{U})(K'\cap T)
\subset K_n \cap \zeta K_n \zeta^{-1}$.

Suppose that $n \geq 2$.
For $k \in K'$,
we set $y = k_{22}^{-1}k_{12}$.
Then $y$ lies in  $\ri_E$.
Since $\det k \in E^{1}$, the element
$k_{33}$ must belong to $\ri_E^\times$.
As in the proof of (i),
we can take 
$x \in  \mi_F^{1-n}$
such that 
$k_{13} -x\e k_{33}\in \ri_E \subset \mi_E^{2-n}$.
Then one can easily check that
$k$ belongs to $u(y, x\e -y\overline{y}/2)(K_n \cap \zeta K_n \zeta^{-1})$.
This completes the proof.
\end{proof}

Lemma~\ref{lem:hecke_coset}
gives an explicit description of the Hecke operator $T$.
For $v \in V(n)$,
we set
\begin{eqnarray}
v' = \sum_{y\in \mi_E^{n-1}/\mi_E^{n}}
\sum_{z\in \mi_F^{n-1}/\mi_F^{n}}
\pi(\hat{u}(y, z\e-y\overline{y}/2)) v.
\end{eqnarray}
Then we get the following
\begin{lem}\label{lem:hecke1}
Suppose that $n \geq 1$.
Then we have
\[
Tv = \pi(\zeta^{-1}) v' +
\sum_{a \in \ri_E/\mi_E}\sum_{b \in \mi_F^{-n}/\mi_F^{2-n}}
\pi({u}(a, b\e-a\overline{a}/2)\zeta ) v,
\]
for $v \in V(n)$.
\end{lem}
\begin{proof}
By (\ref{eq:T}) and Lemma~\ref{lem:hecke_coset},
we obtain
\begin{eqnarray*}
Tv & =& \sum_{\substack{y \in \ri_E/\mi_E\\z \in \mi_F^{1-n}/\mi_F^{2-n}}}
\pi(t_n{u}(y, z\e-y\overline{y}/2)\zeta ) v
+
\sum_{\substack{a \in \ri_E/\mi_E\\b \in \mi_F^{-n}/\mi_F^{2-n}}}
\pi({u}(a, b\e-a\overline{a}/2)\zeta ) v\\
 & =& \sum_{\substack{y \in \ri_E/\mi_E\\z \in \mi_F^{1-n}/\mi_F^{2-n}}}
\pi(\zeta^{-1}\zeta t_n{u}(y, z\e-y\overline{y}/2)\zeta t_n) v
+
\sum_{\substack{a \in \ri_E/\mi_E\\b \in \mi_F^{-n}/\mi_F^{2-n}}}
\pi({u}(a, b\e-a\overline{a}/2)\zeta ) v\\
 & =& \pi(\zeta^{-1})\sum_{\substack{y \in \ri_E/\mi_E\\z \in \mi_F^{1-n}/\mi_F^{2-n}}}
\pi(t_{n-1}{u}(y, z\e-y\overline{y}/2) t_{n-1}) v
+
\sum_{\substack{a \in \ri_E/\mi_E\\b \in \mi_F^{-n}/\mi_F^{2-n}}}
\pi({u}(a, b\e-a\overline{a}/2)\zeta ) v\\
 & =&
\pi(\zeta^{-1}) \sum_{\substack{y \in \mi_E^{n-1}/\mi_E^{n}\\z \in \mi_F^{n-1}/\mi_F^{n}}}
\pi(\hat{u}(y, z\e-y\overline{y}/2)) v+
\sum_{\substack{a \in \ri_E/\mi_E\\b \in \mi_F^{-n}/\mi_F^{2-n}}}
\pi({u}(a, b\e-a\overline{a}/2)\zeta ) v\\
 & =&
 \pi(\zeta^{-1}) v' +
\sum_{\substack{a \in \ri_E/\mi_E\\b \in \mi_F^{-n}/\mi_F^{2-n}}}
\pi({u}(a, b\e-a\overline{a}/2)\zeta ) v,
\end{eqnarray*}
as required.
\end{proof}

We shall consider the case when 
$n = N_\pi$.
Because $V(N_\pi)$ is one-dimensional,
there exists $\lambda \in \C$ such that
$T v = \lambda v$ for all $v\in V(N_\pi)$.
We call $\lambda$ {\it the Hecke eigenvalue of $T$}.
For a newform $v$ in $V(N_\pi)$,
we put
\begin{eqnarray}\label{eq:c}
c_i = W_v(\zeta^i),\ c'_i = W_{v'}(\zeta^i),\ i \in \Z.
\end{eqnarray}
Then we obtain the following
\begin{lem}\label{lem:hecke_rec}
Suppose that $N_\pi \geq 1$.
Then we have
\[
\lambda c_i = c'_{i-1} + q^4 c_{i+1},\ i \geq 0.
\]
\end{lem}
\begin{proof}
Set $n= N_\pi$.
By Lemma~\ref{lem:hecke1},
we have
\begin{eqnarray*}
\lambda W_v(\zeta^i) & = & W_{\lambda v}(\zeta^i) =  W_{Tv}(\zeta^i)\\
& = & 
W_{v'}(\zeta^{i-1}) +
\sum_{\substack{a \in \ri_E/\mi_E\\b \in \mi_F^{-n}/\mi_F^{2-n}}}
W(\zeta^i {u}(a, b\e-a\overline{a}/2)\zeta ) \\
 & = & 
W_{v'}(\zeta^{i-1}) +
\sum_{\substack{a \in \ri_E/\mi_E\\b \in \mi_F^{-n}/\mi_F^{2-n}}}
W_v(\zeta^{i}{u}(a, b\e-a\overline{a}/2)\zeta^{-i}\zeta^{i+1} ) \\
 & = & 
W_{v'}(\zeta^{i-1}) +
\sum_{\substack{a \in \mi_E^i/\mi_E^{i+1}\\b \in \mi_F^{2i-n}/\mi_F^{2i+2-n}}}
W_v ({u}(a, b\e-a\overline{a}/2)\zeta^{i+1}) \\
 & = & 
W_{v'} (\zeta^{i-1}) +
\sum_{\substack{a \in \mi_E^i/\mi_E^{i+1}\\b \in \mi_F^{2i-n}/\mi_F^{2i+2-n}}}
\psi_E(a)
W_v (\zeta^{i+1})\\
& = & 
W_{v'} (\zeta^{i-1}) +
q^2\sum_{a \in \mi_E^i/\mi_E^{i+1}}
\psi_E(a)
W_v (\zeta^{i+1}),
\end{eqnarray*}
for $i \in \Z$.
Since we are assuming that $\psi_E$ has conductor $\ri_E$,
we obtain
\begin{eqnarray*}
\lambda W_v(\zeta^i) & = & 
W_{v'}(\zeta^{i-1}) +
q^4
W_v (\zeta^{i+1}),
\end{eqnarray*}
for $i \geq 0$.
This proves the lemma.
\end{proof}

\subsection{Level lowering operator}\label{subsec:hecke2}
Let $(\pi, V)$ be an irreducible generic representation of $G$
and let
$n$ be an integer greater than $n_\pi$.
Then $Z_{n-1}$ acts on $V$ trivially,
and hence every vector in $V(n)$ is fixed by
$Z_{n-1}K_n$.
We define {\it the level lowering operator} $\delta:
V(n) \rightarrow V(n-1)$ by
\[
\delta v = 
\frac{1}{\mathrm{vol}(K_{n-1}\cap (Z_{n-1}K_n))}
\int_{K_{n-1}}\pi(k)v dv,\ v \in V(n).
\]
By the assumption $n > n_{\pi}$,
we can write $\delta v$ as 
\begin{eqnarray}\label{eq:delta}
\delta v 
= \sum_{K_{n-1}/K_{n-1}\cap (Z_{n-1}K_n)}\pi(k)v,\
v \in V(n).
\end{eqnarray}

\begin{lem}\label{lem:level_coset}
Suppose that $n \geq 2$.
Then a complete set of representatives for 
$K_{n-1}/K_{n-1}\cap (Z_{n-1}K_n)$
is given by
\[
t_{n-1} \hat{u}(y, -y\overline{y}/2)\ {and}\ \hat{u}(a, b\e-a\overline{a}/2),
\]
where $y, a \in \mi_E^{n-1}/\mi_E^n$
and  $b \in \mi_F^{n-1}/\mi_F^n$.
\end{lem}
\begin{proof}
One can easily check that
\[
K_{n-1}\cap (Z_{n-1}K_n)
 =
 \left(
\begin{array}{ccc}
\ri_E & \ri_E &  \mi_E^{1-n}\\
\mi_E^{n} & 1+ \mi_E^{n-1} &  \ri_E\\
\mi_E^{n} & \mi_E^{n} &  \ri_E
\end{array}
\right)\cap G.
\]
We set 
\[
K''
 =
 \left(
\begin{array}{ccc}
\ri_E & \ri_E &  \mi_E^{1-n}\\
\mi_E^{n-1} & 1+ \mi_E^{n-1} &  \ri_E\\
\mi_E^n & \mi_E^{n-1} &  \ri_E
\end{array}
\right)\cap G.
\]
Then it suffices to show the following two claims:
\begin{enumerate}
\item[(i)]
We can take
a complete set of representatives for 
$K_{n-1} / K''$
as the $q+1$ elements
$t_{n-1}$ {and} $\hat{u}(0, x\e)$,
where 
$x \in \mi_F^{n-1}/\mi_F^{n}$.

\item[(ii)]
A complete set of representatives for 
$K''/K_{n-1}\cap (Z_{n-1}K_n)$
is given by the $q^2$ elements
$\hat{u}(y, -y\overline{y}/2)$,
where $y \in \mi_E^{n-1}/\mi^n_E$.
\end{enumerate}
We shall prove these claims.
Clearly, the
elements in (i) and (ii) are contained in pairwise disjoint cosets
in $K_{n-1} / K''$ and $K''/K_{n-1}\cap (Z_{n-1}K_n)$
respectively.

(i)
Let $k \in K_{n-1}$.
If $k_{11} \in \mi_E$,
then we have
$t_{n-1} k \in K''$. So we get $k \in t_{n-1} K''$.
Suppose that $k_{11} \in \ri_E^\times$.
Since $k$ lies in $G$,
we have $k_{11}\overline{k}_{31}+k_{21}\overline{k}_{21}
+k_{31}\overline{k}_{11} = 0$.
This implies 
$k_{11}\overline{k}_{31}
+k_{31}\overline{k}_{11} \in \mi^{2n-2}_F$,
and hence
$k_{31}\overline{k}_{11} \in \mi^{2n-2}_F \oplus \mi_F^{n-1}\e$.
By the assumption $k_{11} \in \ri_E^\times$,
we have
$k_{31}{k}_{11}^{-1} \in \mi^{2n-2}_F \oplus \mi_F^{n-1}\e$.
So there exists $x \in  \mi_F^{n-1}$
such that 
$k_{31}{k}_{11}^{-1} -x\e \in \mi^{2n-2}_F$.
Using the assumption again,
we get 
$k_{31} -x\e k_{11}\in \mi^{2n-2}_E \subset \mi_E^{n}$.
This implies $k \in \hat{u}(0, x\e)K''$.

(ii)
Let $k \in K''$.
Since $\det k$ lies in $E^{1}$,
we have
$k_{11} \in \ri_E^\times$.
Set $y = k_{11}^{-1}k_{21} \in \mi_E^{n-1}$.
Then one can easily check that
$k$ belongs to $\hat{u}(y, -y\overline{y}/2)(K_{n-1}\cap (Z_{n-1}K_n))$.
This completes the proof.
\end{proof}

By Lemma~\ref{lem:level_coset},
we get an explicit formula of the level lowering operator $\delta$.
\begin{lem}\label{lem:level1}
Suppose that $n \geq 2$ and $n > n_\pi$.
Then we have
\[
\delta v = v' + 
 \sum_{y \in \mi_E^{-1}/\ri_E}
\pi( \zeta {u}(y, -y\overline{y}/2)) v,
\]
for $v \in V(n)$.
\end{lem}
\begin{proof}
By (\ref{eq:delta}) and Lemma~\ref{lem:level_coset},
we get
\begin{eqnarray*}
\delta v & =& \sum_{y \in \mi_E^{n-1}/\mi_E^n}
\pi( t_{n-1} \hat{u}(y, -y\overline{y}/2)) v
+
\sum_{\substack{a \in \mi_E^{n-1}/\mi_E^n\\b \in \mi_F^{n-1}/\mi_F^n}}
\pi(\hat{u}(a, b\e-a\overline{a}/2) ) v\\
& =& \sum_{y \in \mi_E^{n-1}/\mi_E^n}
\pi( \zeta t_{n} \hat{u}(y, -y\overline{y}/2)t_n) v
+
v'\\
& =& \sum_{y \in \mi_E^{-1}/\ri_E}
\pi( \zeta{u}(y, -y\overline{y}/2)) v
+
v',
\end{eqnarray*}
as required.
\end{proof}

We consider the level lowering operator
for $n = N_\pi$.
Since  $V(N_\pi-1) = \{0\}$,
we have
$\delta v = 0$ for all $v \in V(N_\pi)$.
For a newform $v$ in $V(N_\pi)$,
we set $c_i$ and $c'_i$ as in (\ref{eq:c}).
Then we get another relation between $c_i$ and $c'_i$.
\begin{lem}\label{lem:level_rec}
Suppose that $N_\pi \geq 2$ and $N_\pi > n_\pi$.
Then we have
\begin{eqnarray*}
& c'_i + q^2 c_{i+1} = 0, \ i \geq 0,\\
& c'_{-1} = 0.
\end{eqnarray*}
\end{lem}
\begin{proof}
By Lemma~\ref{lem:level1},
we have
\begin{eqnarray*}
0 & = & 
W_{\delta v}(\zeta^i)\\
& = & W_{v'}(\zeta^i) + 
 \sum_{y \in \mi_E^{-1}/\ri_E}
W_v ( \zeta^{i+1} {u}(y, -y\overline{y}/2)) \\
& = & W_{v'}(\zeta^i) + 
 \sum_{y \in \mi_E^{i}/\mi_E^{i+1}}
W_v({u}(y, -y\overline{y}/2)\zeta^{i+1}) \\
& = & W_{v'}(\zeta^i) + 
 \sum_{y \in \mi_E^{i}/\mi_E^{i+1}}
 \psi_E(y)
W_v(\zeta^{i+1}).
\end{eqnarray*}
Since $\psi_E$ has conductor $\ri_E$,
we obtain
\begin{eqnarray*}
0 & = & 
W_{v'}(\zeta^{i}) +
q^2
W_v (\zeta^{i+1}),\ \mathrm{for}\ i \geq 0 
\end{eqnarray*}
and
\[
0  =  W_{v'} (\zeta^{-1}).
\]
This proves the lemma.
\end{proof}
\subsection{Zeta integrals of newforms}\label{subsec:hecke3}
Let $(\pi, V)$ be an irreducible generic representation
of $G$ such that $N_\pi \geq 2$ and $N_\pi > n_\pi$.
For $v \in V(N_\pi)$,
we set 
\[
c_i = W_v(\zeta^i),\ i \in \Z.
\]
Then we get a recursion formula for $\{c_i\}_{i \in \Z}$.
\begin{lem}\label{lem:rec}
Suppose that $N_\pi \geq 2$ and $N_\pi > n_\pi$.
Then we have
\begin{eqnarray*}
& (\lambda+q^2)c_i = q^4 c_{i+1}, \ i \geq 1,\\
& \lambda c_{0} = q^4 c_{1}.
\end{eqnarray*}
\end{lem}
\begin{proof}
The assertion follows from Lemmas~\ref{lem:hecke_rec}
and \ref{lem:level_rec}.
\end{proof}

Now we get an explicit  formula of zeta integrals 
of newforms
in terms of Hecke eigenvalues.
\begin{prop}\label{prop:zeta}
Suppose that an irreducible generic representation 
$(\pi, V)$of $G$
satisfies $N_\pi \geq 2$ and $N_\pi > n_\pi$.
Then
we have
\[
Z(s, W_v)
=
\frac{1-q^{-2s}}{\displaystyle 1-\frac{\lambda+q^2}{q^2}q^{-2s}}W_v(1),
\]
for $v \in V(N_\pi)$,
where $\lambda$ is the eigenvalue of the Hecke operator 
$T$ on $V(N_\pi)$.
\end{prop}
\begin{proof}
Since $v$ is fixed by $t(b)$, $b \in \ri_E^\times$,
we get 
$W_v(t(a)) = W_v(\zeta^{\nu_E(a)})$, for $a \in E^\times$,
where $\nu_E$ is the valuation on $E$
normalized so that $\nu_E(\p) = 1$.
Proposition~\ref{prop:new} (i)
says that 
$W_v(\zeta^i) = 0$ for all $i < 0$.
Then, by Lemma~\ref{lem:rec},
we obtain
\begin{eqnarray*}
Z(s, W_v) & = & 
\int_{E^\times}W_v(t(a))|a|_E^{s-1}d^\times a
 = 
\sum_{i = 0}^\infty W_v(\zeta^i) |\p^i|_E^{s-1}\\
&  = &
\sum_{i = 0}^\infty c_i q^{2i(1-s)}\\
& = & 
c_0 + c_0\frac{\lambda}{q^4}\sum_{i = 0}^\infty \left(\frac{\lambda+q^2}{q^4}\right)^i q^{2(i+1)(1-s)}\\
& = &
c_0+c_0\frac{\lambda}{q^2}q^{-2s}\sum_{i = 0}^\infty
\left(\frac{\lambda+q^2}{q^2}\right)^i q^{-2is}\\
& = &
c_0+c_0\frac{\lambda}{q^2}q^{-2s}
\frac{1}{1-\displaystyle \frac{\lambda+q^2}{q^2}q^{-2s}}\\
& = & 
\frac{c_0(1-q^{-2s})}{1-\displaystyle \frac{\lambda+q^2}{q^2}q^{-2s}}.
\end{eqnarray*}
This shows the proposition.
\end{proof}

\Section{Proof of Theorem~\ref{thm:main}}\label{pf}
Now we shall prove Theorem~\ref{thm:main}.
Let $(\pi, V)$ be an irreducible generic supercuspidal representation of $G$.
By Theorem~\ref{thm:new} (iii),
we can apply Proposition~\ref{prop:zeta}
to newforms for $\pi$.
Due to Proposition~\ref{prop:new} (ii),
we can take $v \in V(N_\pi)$ such that $W_v(1) = 1$.
For such $v$, 
we have
\[
Z(s, W_v)
=
\frac{1-q^{-2s}}{\displaystyle 1-\frac{\lambda+q^2}{q^2}q^{-2s}}. 
\]
It follows from the proof of Proposition~\ref{prop:L_sc}
that
$Z(s, W_v)$ is a polynomial in $q^{-2s}$ and $q^{2s}$.
Therefore we must have $Z(s, W_v) = 1-q^{-2s}$ or $1$,
and hence
$Z(s, W_v, \Phi_{N_\pi}) = Z(s, W_v)L_E(s, 1) =  1$ or $L_E(s, 1)$ 
by Proposition~\ref{prop:zeta1}.

Suppose that $Z(s, W_v, \Phi_{N_\pi}) \neq L(s, \pi)$.
Then by Proposition~\ref{prop:L_sc},
we must have
$Z(s, W_v, \Phi_{N_\pi}) = 1$ and $L(s, \pi) = L_E(s, 1)$.
Thus, the functional equation
\[
\frac{Z(1-s, W_v, \hat{\Phi}_{N_\pi})}{L(1-s, \pi)}
= \varepsilon(s, \pi, \psi_F, \psi_E)
\frac{Z(s, W_v, {\Phi}_{N_\pi})}{L(s, \pi)}
\]
implies
\[
q^{-2N_{\pi}(s-1/2)}\frac{1}{L_E(1-s, 1)}
= \varepsilon(s, \pi, \psi_F, \psi_E)
\frac{1}{L_E(s, 1)}
\]
because of Proposition~\ref{prop:zeta2}.
This contradicts Proposition~\ref{prop:mono},
which states that $\varepsilon(s, \pi, \psi_F, \psi_E)$ is monomial.
We therefore conclude that $Z(s, W_v, \Phi_{N_\pi}) = L(s, \pi)$.
This completes the proof
of Theorem~\ref{thm:main}.

\end{document}